\documentclass[11pt,a4paper]{amsart}

\usepackage[T1]{fontenc}
\usepackage{lmodern}
\usepackage{mathtools,amssymb}
\usepackage[margin=29mm]{geometry}
\usepackage{microtype}
\usepackage{booktabs}
\usepackage[hidelinks]{hyperref}

\hypersetup{
  pdftitle={Index and nullity of critical spherical catenoids in hyperbolic three-space},
  pdfauthor={Ivan Dorofeev and Alexander Pigazzini},
  pdfsubject={Morse index four and nullity two for critical spherical catenoids},
  pdfkeywords={free boundary minimal surfaces, hyperbolic catenoids, Morse index, nullity, Robin spectrum},
  bookmarksnumbered=true
}

\newtheorem{theorem}{Theorem}
\newtheorem{proposition}[theorem]{Proposition}
\newtheorem{lemma}[theorem]{Lemma}
\newtheorem{corollary}[theorem]{Corollary}
\theoremstyle{remark}
\newtheorem{remark}[theorem]{Remark}

\newcommand{\HH}{\mathbb H}
\newcommand{\RR}{\mathbb R}
\newcommand{\II}{\mathrm{II}}
\newcommand{\dd}{\,\mathrm d}
\newcommand{\ind}{\operatorname{ind}}
\newcommand{\nul}{\operatorname{nul}}
\newcommand{\arcosh}{\operatorname{arcosh}}
\newcommand{\tr}{\operatorname{tr}}
\newcommand{\Span}{\operatorname{span}}
\newcommand{\qk}{\mathfrak q}

\numberwithin{equation}{section}
\title[Index and nullity of critical catenoids]
{Index and nullity of critical spherical catenoids\\
in hyperbolic three-space}
\author{Ivan Dorofeev}
\author{Alexander Pigazzini}
\date{}
\subjclass[2020]{Primary 53A10; Secondary 53C42, 58J50, 34B24}
\keywords{Free boundary minimal surfaces, hyperbolic space, critical catenoid, Morse index, nullity, Robin eigenvalue problem, equivariant nondegeneracy, constant mean curvature, energy index.}

\begin{document}

\begin{abstract}
We prove that every critical spherical catenoid in a geodesic ball of hyperbolic three-space has Morse index four and nullity two. This resolves Medvedev's index conjecture throughout the parameter range of the catenoid family. The kernel consists precisely of the normal components of rotations
that preserve the center of the ball and tilt the axis of the catenoid.
As consequences, the radius of the supporting ball increases strictly along the family, and each critical spherical catenoid admits a smooth family of free boundary constant mean curvature deformations in its supporting ball. The proof combines a global boundary inequality obtained from a Wronskian, a positive Dirichlet eigenfunction supplied by the time component of the unit normal, and a Riccati comparison for the angular mode $k=2$.
A ground-state identity then yields positivity of every angular mode $k\ge2$.
We also show that the energy index of every critical spherical catenoid in $\HH^3$, that is, the index of the second variation of its energy, equals four; hence the inequality between the energy index and the spectral index conjectured by Medvedev for geodesic balls of $\HH^n$ fails.
\end{abstract}

\maketitle

\section{Introduction}

The critical catenoid is a basic example in the study of the Morse index of free boundary minimal surfaces. In the Euclidean unit ball in $\RR^3$, its index is four, as proved by Devyver \cite{Devyver}, Smith and Zhou \cite{SmithZhou}, and Tran \cite{Tran}; another proof was given by Medvedev \cite{MedvedevIndex}. Devyver also proved that every compact orientable free boundary minimal surface in this ball, other than a flat disk, has index at least four \cite{Devyver}. Thus the critical catenoid attains this lower bound. Its nullity is two, and it is equivariantly nondegenerate: its Jacobi fields satisfying the free boundary condition are induced by rotations \cite[Lemma~8.1]{BettiolPiccioneSantoro}; see also \cite[Remark~8.2]{BettiolPiccioneSantoro}. Moreover, Parise and Zhu have recently shown that every properly embedded free boundary minimal annulus in this ball is congruent to the critical catenoid \cite{PariseZhu}.

The analogous index problem in curved space forms concerns critical spherical catenoids. In $\HH^3$, these are rotational minimal annuli obtained by free boundary truncation of the spherical catenoids constructed by Mori \cite{Mori}. The geometry of rotational hypersurfaces in spaces of constant curvature was developed further by do Carmo and Dajczer \cite{doCarmoDajczer}, and the free boundary truncation was
considered by Li and Xiong \cite{LiXiong}. C.~Lima \cite{Lima} proved that an embedded free boundary minimal annulus in a geodesic ball of $\mathbb S^3_+$ or $\HH^3$ that is invariant under the geodesic symmetry about the center of the ball is congruent to a critical rotational annulus; in $\HH^3$, these are the critical spherical catenoids. Medvedev \cite[Theorem~2]{Medvedev}
proved that the critical spherical catenoid in a geodesic ball of $\mathbb S^3_+$ has index four, and that its hyperbolic counterpart has index at least four. He conjectured equality in the hyperbolic case \cite[Remark~10]{Medvedev}.
We prove this equality throughout the parameter range of the critical spherical catenoid family in $\HH^3$ and determine the nullity.

The analysis in \cite{PigRobin} determines the index and nullity contributions of the angular mode $|k|=1$ for every parameter. The local result in \cite{PigLocal} establishes index four and nullity two when the parameter is sufficiently close to its lower endpoint. The theorem below establishes both equalities for every $a>\tfrac12$.

\begin{theorem}\label{thm:main}
For each $a>\tfrac12$, let $\Sigma_a$ be the critical spherical catenoid with angular period $2\pi$, parametrized in Section~\ref{sec:geometry}, and let $\mathbb B(r(a))\subset\HH^3$ be its supporting geodesic ball.
Then
\[
\ind(\Sigma_a)=4,\qquad \nul(\Sigma_a)=2.
\]
The Robin kernel is generated by the normal components of the two rotations that tilt the axis of the catenoid while fixing the center of $\mathbb B(r(a))$.
\end{theorem}

Here the index and nullity refer to the second variation of area for normal variations whose boundary remains on the supporting sphere.
In particular, the ambient ball is fixed when taking variations.

The theorem supplies the global hyperbolic counterpart of the Euclidean index-four results and identifies all infinitesimal degeneracies with rotations of the ambient ball. Together with the Euclidean results and Medvedev's spherical result, it shows that critical catenoids have index four in the Euclidean unit ball, in geodesic balls of $\mathbb S^3_+$, and in every geodesic ball of $\HH^3$. At fixed ambient curvature $-1$, both the geometry of the annulus and its supporting radius vary with $a$.
The associated Jacobi operator contains the curvature term $+2$, and the Robin coefficient is $\coth r(a)$; these terms are recorded in Section~\ref{sec:spectral}. The proof controls this family of boundary value problems over the entire interval $a>\tfrac12$.

The proof has three main ingredients. A Wronskian identity gives a strict boundary inequality valid for every $a>\tfrac12$. The time component of the unit normal is a positive Dirichlet eigenfunction with eigenvalue two; a codimension-two min--max argument therefore determines the radial mode $k=0$. Finally, a comparison of logarithmic
derivatives, together with an explicit integral inequality, gives a boundary inequality for the mode $k=2$; a ground-state identity turns it into positivity of that mode, and monotonicity in $k^2$ completes the spectral count. The geometric and spectral identities used below are included in the proof.

Three consequences are recorded in Section~\ref{sec:consequences}.
The radius of the supporting ball is a strictly increasing function of $a$ (Corollary~\ref{cor:radius}). Together with Lemma~\ref{lem:radiuslimit}, this proves that every geodesic ball of $\HH^3$ is the supporting ball of an isometric image of exactly one $\Sigma_a$, as stated without proof by Li and Xiong \cite{LiXiong}. Combined with the equivariant implicit function theorem of Bettiol, Piccione and Santoro~\cite{BettiolPiccioneSantoro}, the theorem shows that each $\Sigma_a$ is the member with mean curvature zero of a smooth family of free boundary CMC annuli in its supporting ball (Corollary~\ref{cor:equivariant}). The restriction to angular period $2\pi$ is essential: for $m\ge2$, the $m$-fold cover of $\Sigma_a$ has index at least $4m$, and its nullity exceeds two exactly when $mk_*(a)$ is an integer, for a threshold frequency $k_*(a)\in(1,2)$ (Remark~\ref{rem:covers}). This agrees with the non-rotational free boundary minimal annuli of Cerezo~\cite{Cerezo}, which bifurcate from multiple covers of critical spherical catenoids.

Medvedev also compared the energy index $\ind_E$, the index of the second variation of the energy at fixed conformal structure, with the spectral index $\ind_S$ (denoted $\operatorname{Ind}_E$ and $\operatorname{Ind}_S$ in \cite{Medvedev}; both are recalled in Section~\ref{sec:energy}). For free boundary minimal submanifolds of geodesic balls in $\mathbb S^n_+$ he proved $\ind_E\le n\ind_S$, where $n$ is the dimension of the ambient space, and he conjectured the same inequality in $\HH^n$ \cite[Proposition~11 and Remark~7]{Medvedev}. For annuli, V.~Lima's bound $\ind\le\ind_E+1$ \cite{LimaBounds} then gives $\ind\le n\ind_S+1$; in $\mathbb S^3_+$ this is the upper bound for the index of the critical spherical catenoid \cite[Theorem~12 and Corollary~17]{Medvedev}, since $\ind_S=1$ there. In $\HH^3$ one also has $\ind_S(\Sigma_a)=1$ \cite[Theorem~18]{Medvedev} (see also Section~\ref{sec:energy}), so the conjectured inequality would give $\ind(\Sigma_a)\le4$ in the same way. Our second result shows that the conjectured inequality fails in $\HH^3$.

\begin{theorem}\label{thm:energy}
For every $a>\tfrac12$, the energy index of $\Sigma_a$ equals four. In particular $\ind_E(\Sigma_a)>3\ind_S(\Sigma_a)$, and the inequality conjectured in \cite[Remark~7]{Medvedev} fails for every critical spherical catenoid in $\HH^3$.
\end{theorem}

In $\mathbb S^3_+$ the energy index of the critical spherical catenoid is three, by \cite[Theorem~2, Proposition~11 and Theorem~18]{Medvedev} and V.~Lima's bound; in $\HH^3$ it equals the Morse index. Consequently, the upper bound in Theorem~\ref{thm:main} cannot be obtained as in the spherical case: V.~Lima's bound only gives $\ind(\Sigma_a)\le5$. The fourth negative direction is explicit: it comes from the quadric cone with vertex $e_0$ through the two boundary circles (Lemma~\ref{lem:conical}). For the critical catenoid in the Euclidean unit ball the corresponding field is null for the energy, and negative curvature makes it strictly negative (Remark~\ref{rem:curvature}). The proof is given in Section~\ref{sec:energy}.

\section{Geometry of the critical catenoids}\label{sec:geometry}

We use the hyperboloid model
\[
\HH^3=\{X\in\RR^{3,1}:\langle X,X\rangle_L=-1,\ X^0>0\},
\qquad
\langle X,Y\rangle_L=-X^0Y^0+\sum_{j=1}^3X^jY^j.
\]
The center of the balls under consideration is $e_0=(1,0,0,0)$.
Fix $a>\tfrac12$ and define
\begin{align}
K&=\sqrt{a^2-\tfrac14},&
A(s)^2&=a\cosh(2s)+\tfrac12,&
B(s)^2&=a\cosh(2s)-\tfrac12,\label{eq:AB}\\
\varphi(s)&=\int_0^s\frac{K}{A(\sigma)^2B(\sigma)}\dd\sigma.
\label{eq:phi}
\end{align}
The positive square roots are understood. The rotational parametrization is
\begin{equation}\label{eq:X}
X(s,\theta)=
\bigl(A\cosh\varphi,\ A\sinh\varphi,\ B\cos\theta,\ B\sin\theta\bigr),
\qquad
(s,\theta)\in\RR\times(\RR/2\pi\mathbb Z).
\end{equation}
A prime denotes differentiation with respect to $s$ throughout the paper. 
Set $v=B'$. The identities
\begin{gather}
A^2=1+B^2,\qquad A'=\frac{Bv}{A},\qquad
v=\frac{a\sinh(2s)}{B},\label{eq:basic}\\
v^2=1+B^2-\frac{K^2}{B^2},\qquad
v'=B+\frac{K^2}{B^3}>0
\label{eq:Bident}
\end{gather}
follow directly from \eqref{eq:AB}. In particular, $A$ and $B$ are even, whereas $v$ and $\varphi$ are odd.

For completeness, introduce the Lorentz-orthonormal frame
\[
E_0=(\cosh\varphi,\sinh\varphi,0,0),\quad
E_1=(\sinh\varphi,\cosh\varphi,0,0),\quad
E_r=(0,0,\cos\theta,\sin\theta),
\]
together with $E_\theta=(0,0,-\sin\theta,\cos\theta)$.
Then $X=AE_0+BE_r$ and
\[
X_s=A'E_0+A\varphi'E_1+vE_r,\qquad X_\theta=BE_\theta.
\]
Using \eqref{eq:phi}--\eqref{eq:Bident}, we obtain
\[
\langle X_s,X_s\rangle_L
=\frac{v^2+K^2/B^2}{A^2}=1,\qquad
\langle X_s,X_\theta\rangle_L=0.
\]
Consequently,
\begin{equation}\label{eq:metric}
g=ds^2+B^2d\theta^2.
\end{equation}
A unit normal is
\begin{equation}\label{eq:normal}
\nu=\frac KA E_0-\frac vA E_1+\frac KB E_r.
\end{equation}
Indeed, \eqref{eq:phi}--\eqref{eq:Bident} imply that $\nu$ is orthogonal to $X,X_s,X_\theta$ and has Lorentzian norm one. A direct computation using \eqref{eq:phi}--\eqref{eq:Bident} gives
\[
\nu_s=-\frac{K}{B^2}X_s,\qquad
\nu_\theta=\frac{K}{B^2}X_\theta.
\]
Thus, with shape operator $-d\nu$, the principal curvatures are
$K/B^2$ and $-K/B^2$. The surface is minimal and
\begin{equation}\label{eq:II}
|\II|^2=\frac{2K^2}{B^4}.
\end{equation}
The parametrization is embedded: the quotient
$X^1/X^0=\tanh\varphi(s)$ is strictly increasing in $s$, and, for each fixed $s$, the pair $(X^2,X^3)=B(s)(\cos\theta,\sin\theta)$ determines $\theta$ uniquely modulo $2\pi$ because $B(s)>0$.

\subsection{Free boundary truncation}

Write
\[
g_0(s)=A(s)\cosh\varphi(s)=X^0,\qquad
g_1(s)=A(s)\sinh\varphi(s)=X^1.
\]
For the distance $\rho$ from $e_0$, one has $X^0=\cosh\rho$ and
\[
\nabla^{\HH^3}X^0=X^0X-e_0,\qquad
\langle\nabla^{\HH^3}X^0,\nu\rangle_L=\nu^0.
\]
It follows that the surface meets a sphere centered at $e_0$
orthogonally precisely when $\nu^0=0$. At a point $s=S>0$, this
condition reads
\begin{equation}\label{eq:FB}
\tanh\varphi(S)=\frac{K}{v(S)}.
\end{equation}
As functions of $S>0$, the left-hand side increases strictly from zero, whereas the right-hand side decreases strictly from $+\infty$ to zero. Hence \eqref{eq:FB} has exactly one solution $S>0$.

The critical catenoid is
\[
\Sigma_a=X([-S,S]\times(\RR/2\pi\mathbb Z)).
\]
It is symmetric with respect to the totally geodesic plane
$\{X^1=0\}$, which contains the neck circle $s=0$, and its
supporting ball is centered at the point $e_0$ where this plane meets the axis $\{X^2=X^3=0\}$.
Set
\begin{equation}\label{eq:boundarynotation}
B_0=B(S),\qquad v_0=v(S),\qquad
r=\arcosh g_0(S),\qquad c=\coth r.
\end{equation}
Since $g_0'(s)>0$ for $s>0$, every interior point of $\Sigma_a$
lies strictly inside $\mathbb B(r)$.
At the positive boundary the outward unit conormal is $X_s$.
Orthogonality to the sphere gives
\[
X_s(S,\theta)=\frac{\cosh r\,X(S,\theta)-e_0}{\sinh r}.
\]
Taking components yields
\begin{gather}
g_0(S)=\cosh r,\qquad g_0'(S)=\sinh r,\qquad
g_1'(S)=c\,g_1(S),\label{eq:boundarycoordinates}\\
v_0=cB_0,\qquad
g_0'(S)-c\,g_0(S)=-\frac1{\sinh r}.
\label{eq:boundarycoefficient}
\end{gather}

\section{The quadratic form and separation of variables}\label{sec:spectral}

We use the convention $\Delta=\operatorname{div}\nabla$.
For the ambient Ricci tensor, $\operatorname{Ric}_{\HH^3}(\nu,\nu)=-2$.
If $N$ is the outward unit normal to the supporting sphere, then $\langle\nabla^{\HH^3}_V N,V\rangle_L=\coth r$ for every unit vector $V$ tangent to that sphere. Along the free boundary, $\nu$ is such a tangent vector. In the sign convention of \cite{Tran}, the second fundamental form of the boundary is computed with respect to the outward normal, so that its value on $\nu$ is $-\coth r$. The general free boundary second variation formula \cite[Section~2]{Tran} (see also \cite[Section~5.1]{Medvedev}) therefore gives
\begin{equation}\label{eq:Q}
Q[u]=
\int_{\Sigma_a}\bigl(|\nabla u|^2+(2-|\II|^2)u^2\bigr)\dd A
-c\int_{\partial\Sigma_a}u^2\dd\ell ,
\qquad u\in H^1(\Sigma_a).
\end{equation}
Its associated self-adjoint operator and boundary condition are
\[
J=-\Delta-|\II|^2+2,\qquad \partial_\eta u=cu,
\]
where $\eta$ is the outward unit conormal along $\partial\Sigma_a$.
The index is the number of negative eigenvalues, counted with
multiplicity; the nullity is the dimension of the kernel.
The form domain is $H^1(\Sigma_a)$ and imposes no Robin condition on test functions.

For angular frequency $k\ge0$, the radial operator is
\begin{equation}\label{eq:Jk}
J_k u=-\frac1B(Bu')'
+\left(2+\frac{k^2}{B^2}-\frac{2K^2}{B^4}\right)u
\end{equation}
in $L^2((-S,S),B\dd s)$, with boundary conditions
\begin{equation}\label{eq:Robin}
u'(S)=cu(S),\qquad -u'(-S)=cu(-S).
\end{equation}
After removing the positive angular integration factor, the radial form is
\begin{equation}\label{eq:qrad}
\begin{split}
\qk_k[u]={}&
\int_{-S}^{S}
\left(Bu'^2+
\left(2B+\frac{k^2}{B}-\frac{2K^2}{B^3}\right)u^2\right)\dd s\\
&-v_0\bigl(u(S)^2+u(-S)^2\bigr),
\qquad u\in H^1(-S,S).
\end{split}
\end{equation}
Each $k\ge1$ has the two real angular factors $\cos(k\theta)$ and $\sin(k\theta)$; the mode $k=0$ has one. The coefficients are even in $s$, so each radial problem splits into even and odd sectors.

For each fixed $a>\tfrac12$, one has $B(s)\ge\sqrt{a-\tfrac12}>0$ on the compact interval $[-S,S]$.
All radial problems are therefore regular Sturm--Liouville problems with positive smooth weight. We use the min--max principle, simplicity of radial eigenvalues, and the oscillation theorem: the $j$th eigenfunction, with eigenvalues numbered starting at $j=1$, has exactly $j-1$ interior zeros. These standard facts are discussed in \cite{Zettl}.

\section{A global boundary inequality}\label{sec:boundaryinequality}

For the flat connection $D$ of $\RR^{3,1}$, the Gauss formula reads
\[
D_U(dX(V))=dX(\nabla^\Sigma_UV)+\II(U,V)\nu+g(U,V)X.
\]
Tracing and using minimality gives $\Delta_\Sigma X=2X$.
The coordinate functions consequently satisfy
$\Delta_\Sigma X^j=2X^j$, so
\begin{equation}\label{eq:gode}
g_i''+\frac{B'}B g_i'-2g_i=0,\qquad
J_0g_i=-|\II|^2g_i,\qquad i=0,1.
\end{equation}

\begin{lemma}\label{lem:boundary}
For every $a>\tfrac12$, the boundary radius in the rotational
coordinates satisfies
\[
B_0^2>2K^2.
\]
\end{lemma}

\begin{proof}
By \eqref{eq:Bident},
\[
v''=v-\frac{3K^2v}{B^4},
\]
and substitution in \eqref{eq:Jk} gives $J_0v=0$.
Moreover, $v(0)=0$ and $v>0$ on $(0,S]$. Define
\[
W=B(vg_1'-v'g_1).
\]
The equations for $v$ and $g_1$ imply
\[
W(0)=0,\qquad W'=B|\II|^2vg_1>0\quad\hbox{on }(0,S].
\]
Using \eqref{eq:boundarycoordinates}, we find
\[
0<W(S)=-B_0g_1(S)\bigl(v'(S)-cv_0\bigr).
\]
On the other hand, \eqref{eq:Bident} and $c=v_0/B_0$ give
\[
v'(S)-cv_0
=B_0+\frac{K^2}{B_0^3}-\frac{v_0^2}{B_0}
=\frac{2K^2-B_0^2}{B_0^3}.
\]
Since $B_0g_1(S)>0$, the asserted inequality follows.
\end{proof}

\section{The modes zero and one}\label{sec:lowmodes}

\subsection{A Dirichlet eigenfunction}

\begin{lemma}\label{lem:normal}
The first Dirichlet eigenvalue of $J_0$ on $(-S,S)$ is $2$.
A corresponding positive eigenfunction is
\[
n(s)=\nu^0(s)
=\frac{K\cosh\varphi(s)-v(s)\sinh\varphi(s)}{A(s)}.
\]
\end{lemma}

\begin{proof}
The Gauss--Weingarten equations in $\RR^{3,1}$ imply
\begin{equation}\label{eq:normalLap}
\Delta_\Sigma\nu=-|\II|^2\nu
\end{equation}
for a minimal surface in $\HH^3$.
To see the signs explicitly, let $\mathcal A$ denote the shape operator. 
In a geodesic orthonormal frame on the surface, tracing the derivative of $D\nu=-dX\circ\mathcal A$ gives
\[
\Delta_\Sigma\nu
=-dX(\operatorname{div}\mathcal A)
-|\II|^2\nu-(\tr\mathcal A)X.
\]
The Codazzi equation in constant curvature gives
$\operatorname{div}\mathcal A=\nabla(\tr\mathcal A)$, and both trace terms vanish by minimality. This proves \eqref{eq:normalLap}.
Taking its time component gives $J_0n=2n$.

The function $n$ is even and $n(0)=K/A(0)>0$. For $0<s<S$,
the uniqueness of the crossing in \eqref{eq:FB} gives
$\tanh\varphi(s)<K/v(s)$, so $n(s)>0$. At $s=\pm S$,
the free boundary condition gives $n=0$. Thus $n$ is a
Dirichlet eigenfunction with no interior zero. The oscillation
theorem identifies its eigenvalue as the first one.
\end{proof}

\begin{proposition}\label{prop:modezero}
The radial Robin problem for $J_0$ has exactly two negative
eigenvalues and no zero eigenvalue. One negative eigenfunction
is even and the other is odd.
\end{proposition}

\begin{proof}
Integrating \eqref{eq:gode} by parts in \eqref{eq:qrad}, and using
\eqref{eq:boundarycoordinates}--\eqref{eq:boundarycoefficient}, gives
\begin{align}
\qk_0[g_0]
&=-\int_{-S}^{S}B|\II|^2g_0^2\dd s-2B_0\coth r<0,
\label{eq:negativeg0}\\
\qk_0[g_1]
&=-\int_{-S}^{S}B|\II|^2g_1^2\dd s<0.
\label{eq:negativeg1}
\end{align}
Since $g_0$ is even and $g_1$ is odd, their mixed term in the
polarized quadratic form is zero. Hence their span is a
two-dimensional negative subspace.

Let $\lambda_1<\lambda_2<\lambda_3<\cdots$ be the radial Robin
eigenvalues. The preceding subspace gives $\lambda_2<0$.
The Dirichlet form domain $H^1_0(-S,S)$ has codimension two in
$H^1(-S,S)$, since its defining conditions are the two endpoint
traces. Therefore the span of the first three Robin eigenfunctions contains a nonzero $z\in H^1_0(-S,S)$. Lemma~\ref{lem:normal}
and the min--max principle yield
\[
2\int_{-S}^{S}Bz^2\dd s
\le \qk_0[z]
\le \lambda_3\int_{-S}^{S}Bz^2\dd s.
\]
Thus $\lambda_3\ge2$. This proves the exact count and excludes
a kernel. The negative tests \eqref{eq:negativeg0} and
\eqref{eq:negativeg1} supply one negative direction of each parity, so each parity contributes exactly one negative eigenvalue.
\end{proof}

\subsection{The first angular mode}

\begin{proposition}\label{prop:modeone}
The radial Robin problem for $J_1$ has one negative eigenvalue
and a one-dimensional kernel generated by $g_0'$. The negative
eigenfunction is even and the kernel function is odd.
\end{proposition}

\begin{proof}
The identity
\[
\left(\frac{B'}B\right)'
=\frac{2K^2}{B^4}-\frac1{B^2}
\]
shows, upon differentiating \eqref{eq:gode}, that
\begin{equation}\label{eq:derivative}
J_1g_0'=0,\qquad J_1g_1'=0.
\end{equation}
Set $f=g_0'$. The function $f$ is odd, positive on $(0,S]$,
and satisfies $f'(0)=2g_0(0)>0$. At $s=S$,
\[
f'(S)=2g_0(S)-c\,g_0'(S)
=2\cosh r-c\sinh r
=\cosh r=c\,f(S).
\]
Parity gives the other boundary condition in \eqref{eq:Robin}.
Thus $f$ is a Robin eigenfunction with eigenvalue zero and
exactly one interior zero. By the oscillation theorem, zero is
the second radial eigenvalue, and it is simple. There is
exactly one negative eigenvalue.

A first eigenfunction can be chosen strictly positive. Reflection $s\mapsto-s$ preserves the problem, so simplicity makes that eigenfunction even. The kernel generator $f$ is odd.
\end{proof}

\section{Positivity of the higher angular modes}\label{sec:highermodes}

Define the conformal coordinate
\[
t=t(s)=\int_0^s\frac{\dd\sigma}{B(\sigma)},\qquad T=t(S).
\]
A dot denotes differentiation with respect to $t$; primes continue to denote differentiation with respect to $s$. We use the same letters for functions composed with $s(t)$.
Equation \eqref{eq:Bident} gives
\begin{equation}\label{eq:conformal}
\dot B=Bv,\qquad
\dot v=\mathcal D(v):=\sqrt{(v^2-1)^2+4K^2},\qquad
v(0)=0,\quad \mathcal D(0)=2a.
\end{equation}
Indeed, $\dot v=Bv'=B^2+K^2/B^2>0$, and the square of this
expression is $(v^2-1)^2+4K^2$.

In this coordinate the zero-eigenvalue equation is
\begin{equation}\label{eq:zero}
\ddot u=(2v^2+k^2-2)u,
\end{equation}
and the full radial form is
\begin{equation}\label{eq:qconformal}
\qk_k[u]=
\int_{-T}^{T}\bigl(\dot u^2+(2v^2+k^2-2)u^2\bigr)\dd t
-v_0\bigl(u(T)^2+u(-T)^2\bigr).
\end{equation}
The norm in the spectral problem is
\begin{equation}\label{eq:weight}
\int_{-S}^{S}Bu^2\dd s=\int_{-T}^{T}B^2u^2\dd t.
\end{equation}
In particular, the conformal change preserves the form and the
zero equation; the transformed eigenvalue problem has weight $B^2$.

\subsection{A boundary logarithmic derivative}

Let $u_2$ solve
\begin{equation}\label{eq:u2}
\ddot u_2=(2v^2+2)u_2,\qquad u_2(0)=1,\quad \dot u_2(0)=0.
\end{equation}
The coefficient is strictly positive. While $u_2$ is positive,
$\ddot u_2>0$, so $u_2$ increases from $1$ and cannot have a
first zero. Hence $u_2>0$ on $[0,T]$ and $\dot u_2>0$ on $(0,T]$.
Define
\[
p_2=\frac{\dot u_2}{u_2},\qquad \delta_2=p_2-v.
\]

\begin{lemma}\label{lem:logderivative}
For every $a>\tfrac12$, one has $p_2(T)>v_0$.
\end{lemma}

\begin{proof}
The Riccati equation for $p_2$ and \eqref{eq:conformal} give
\begin{equation}\label{eq:delta}
\dot\delta_2+(p_2+v)\delta_2=F_2(v),\qquad
\delta_2(0)=0,\qquad F_2(v)=v^2+2-\mathcal D(v).
\end{equation}
Since $(v^2+2)^2-\mathcal D(v)^2=6v^2+4-4a^2$, we have the exact identity
\begin{equation}\label{eq:F2}
F_2(v)=\frac{6v^2+4-4a^2}{v^2+2+\mathcal D(v)}.
\end{equation}
If $a\le1$, then $F_2(v(t))>0$ for every $t>0$.
An integrating factor in \eqref{eq:delta} yields
$\delta_2(T)>0$.

Suppose now that $a>1$. By \eqref{eq:F2}, $F_2(v)$ is
negative for $v^2<2(a^2-1)/3$ and positive for
$v^2>2(a^2-1)/3$. Since $v$ is strictly increasing, the function
\[
M(t)\delta_2(t)=\int_0^t M(\tau)F_2(v(\tau))\dd\tau,\qquad
M(t)=\exp\left(\int_0^t(p_2+v)\dd\tau\right)>0,
\]
first decreases and then increases, if the sign change is
reached. It follows that
\begin{equation}\label{eq:cross}
\delta_2(T)\le0
\quad\Longrightarrow\quad
\delta_2(t)<0\quad(0<t<T).
\end{equation}
If the sign change is not reached before $T$, the same
implication follows from strict decrease throughout the interval.

For a comparison solution, set $w=g_1'$. From the definitions,
\[
w=A'\sinh\varphi+A\varphi'\cosh\varphi>0
\qquad (s\ge0).
\]
The function $w$ is even, and $J_1w=0$ by \eqref{eq:derivative}.
Consequently, for
\[
p_1=\frac{\dot w}{w},\qquad \delta_1=p_1-v,
\]
we have
\[
\dot p_1=2v^2-1-p_1^2,\qquad p_1(0)=0,
\]
and
\[
\dot\delta_1+(p_1+v)\delta_1=v^2-1-\mathcal D(v)<0,\qquad
\delta_1(0)=0.
\]
The strict sign holds because $K>0$, so $\mathcal D(v)>|v^2-1|$.
Thus $p_1<v$ for $t>0$.

Assume, for contradiction, that $\delta_2(T)\le0$.
By \eqref{eq:cross},
$p_1+p_2<2v$ on $(0,T)$. The difference
$\omega=p_2-p_1$ satisfies
\[
\dot\omega+(p_1+p_2)\omega=3,\qquad \omega(0)=0.
\]
Therefore
\begin{align}
\omega(T)
&=3\int_0^T
\exp\left(-\int_\tau^T(p_1+p_2)\dd t\right)\dd\tau
\notag\\
&>\frac3{B_0^2}\int_0^TB(t)^2\dd t
=\frac3{B_0^2}\int_0^SB(s)\dd s.
\label{eq:lower}
\end{align}
Here we used $\dot B/B=v$; the strict inequality holds on
every integration interval with $\tau<T$.

At the boundary, \eqref{eq:gode} and
$g_1'(S)=(v_0/B_0)g_1(S)$ give
\[
p_1(T)=\frac{B_0g_1''(S)}{g_1'(S)}
=\frac{2B_0^2}{v_0}-v_0.
\]
Since $p_2(T)\le v_0$, equation \eqref{eq:Bident} yields
\begin{equation}\label{eq:upper}
\omega(T)\le v_0-p_1(T)
=\frac{2(B_0^2-K^2)}{B_0^2v_0}.
\end{equation}

We now prove that \eqref{eq:lower} and \eqref{eq:upper} are
incompatible. For $s>0$, define
\[
\mathcal I(s)=3\int_0^sB(\sigma)\dd\sigma
-\frac{2(B(s)^2-K^2)}{v(s)}.
\]
Let $x=B^2$ and $\kappa=K^2$. Direct differentiation gives
\begin{align}
\mathcal I'
&=-B+\frac{2(B^2-\kappa)(B+\kappa/B^3)}{v^2}
\notag\\
&=B\left(-1+
\frac{2(x-\kappa)(x^2+\kappa)}{x(x^2+x-\kappa)}\right)
\notag\\
&=B\,
\frac{(x-2\kappa)(x^2-x+\kappa)}{x(x^2+x-\kappa)}.
\label{eq:factor}
\end{align}
For the last step, the numerator expands as
\[
\begin{aligned}
&-x(x^2+x-\kappa)+2(x-\kappa)(x^2+\kappa)\\
&\qquad=x^3-(1+2\kappa)x^2+3\kappa x-2\kappa^2\\
&\qquad=(x-2\kappa)(x^2-x+\kappa).
\end{aligned}
\]
The denominator in \eqref{eq:factor} is positive for $s>0$,
because $x^2+x-\kappa=xv^2>0$.
Since $a>1$, we have $\kappa>\tfrac34$, and
\[
x^2-x+\kappa=(x-\tfrac12)^2+\kappa-\tfrac14>0.
\]
Set
\begin{equation}\label{eq:sstar}
s_*=\tfrac12\arcosh(2a),\qquad
B(s_*)=\sqrt2K,\qquad v(s_*)=\sqrt2a.
\end{equation}
Lemma~\ref{lem:boundary} implies $S>s_*$.
Equation \eqref{eq:factor} shows that $\mathcal I$ is strictly
increasing on $(s_*,S]$.

The identity
\begin{equation}\label{eq:Bbound}
\left(\frac{B'}B\right)^2
=1+\frac1x-\frac{K^2}{x^2}
=\frac{a^2}{K^2}
-\frac{(x-2K^2)^2}{4K^2x^2}
\le\frac{a^2}{K^2}
\end{equation}
follows by expanding the square and using $a^2=K^2+\tfrac14$.
As $B'\ge0$ for $s\ge0$, integration of
$B'\le(a/K)B$ gives
\[
\int_0^{s_*}B\dd s
\ge\frac Ka\left(\sqrt2K-\sqrt{a-\tfrac12}\right).
\]
Using \eqref{eq:sstar}, we conclude that
\begin{equation}\label{eq:Ipositive}
\mathcal I(s_*)\ge
\frac Ka\left(2\sqrt2K-3\sqrt{a-\tfrac12}\right)>0.
\end{equation}
The final strict inequality follows from
\[
8K^2-9(a-\tfrac12)
=(a-\tfrac12)(8a-5)>0.
\]
Thus $\mathcal I(S)>0$. In contrast,
\eqref{eq:lower}--\eqref{eq:upper} imply
\[
3\int_0^SB\dd s
<\frac{2(B_0^2-K^2)}{v_0},
\]
which says $\mathcal I(S)<0$. This contradiction proves
$\delta_2(T)>0$ also when $a>1$.
\end{proof}

\subsection{A positive quadratic form}

\begin{proposition}\label{prop:higher}
Every radial Robin eigenvalue in angular mode $k\ge2$ is
strictly positive.
\end{proposition}

\begin{proof}
Extend $u_2$ evenly to $[-T,T]$. The coefficient in
\eqref{eq:u2} is even, so the extension is smooth, positive,
and solves the same equation. Its logarithmic derivative
$p_2$ is odd. For any $z\in H^1(-T,T)$, writing $z=u_2h$
and integrating
\[
\dot z^2+(2v^2+2)z^2
=u_2^2\dot h^2+\frac{\mathrm d}{\mathrm dt}(p_2z^2)
\]
gives
\begin{equation}\label{eq:ground}
\begin{split}
\qk_2[z]={}&
\int_{-T}^{T}u_2^2
\left(\frac{\mathrm d}{\mathrm dt}\frac z{u_2}\right)^2\dd t\\
&+\bigl(p_2(T)-v_0\bigr)
\bigl(z(T)^2+z(-T)^2\bigr).
\end{split}
\end{equation}
The identity first follows for smooth functions and extends
to $H^1$ by density and continuity of endpoint traces.
Lemma~\ref{lem:logderivative} makes the boundary coefficient
strictly positive. If the right-hand side vanishes, then
$z/u_2$ is constant and $z(T)=z(-T)=0$, which forces $z=0$.
Therefore $\qk_2[z]>0$ for every nonzero $z\in H^1(-T,T)$.

The minimum of the Rayleigh quotient with the positive weight
$B^2\dd t$ in \eqref{eq:weight} is attained for this regular
problem. Hence its first eigenvalue is strictly positive.
Finally,
\[
\qk_k[z]=\qk_2[z]+(k^2-4)\int_{-T}^{T}z^2\dd t
\qquad(k\ge2)
\]
proves the assertion for all higher modes, on the full form
domain and therefore in both parity sectors.
\end{proof}

\section{Proof of the main theorem}\label{sec:conclusion}

\begin{proof}[Proof of Theorem~\ref{thm:main}]
Propositions~\ref{prop:modezero}, \ref{prop:modeone}, and
\ref{prop:higher} give the following contributions. Angular
multiplicities are included in the table.
\begin{center}
\begin{tabular}{@{}lcc@{}}
\toprule
Sector & Index contribution & Nullity contribution\\
\midrule
$k=0$, even & $1$ & $0$\\
$k=0$, odd & $1$ & $0$\\
$k=1$, even & $2$ & $0$\\
$k=1$, odd & $0$ & $2$\\
$k\ge2$, both parities & $0$ & $0$\\
\bottomrule
\end{tabular}
\end{center}
The angular Fourier decomposition is orthogonal for both the
$L^2$ inner product and the quadratic form. Its partial sums converge
in $H^1(\Sigma_a)$, and hence their boundary traces converge in
$L^2(\partial\Sigma_a)$. Thus the decomposition also holds on the
full form domain. The span of the negative eigenfunctions in the
table has dimension four, and its $L^2$-orthogonal complement is
nonnegative for $Q$. A function in the Robin kernel has each Fourier
coefficient in the corresponding radial kernel. The mode counts
therefore give
\[
\ind(\Sigma_a)=4,\qquad
\ker J=\Span\{g_0'(s)\cos\theta,\ g_0'(s)\sin\theta\}.
\]
In particular, $\nul(\Sigma_a)=2$.

To identify the kernel geometrically, consider the spatial
rotation fields
\[
Z_2(X)=(0,-X^2,X^1,0),\qquad
Z_3(X)=(0,-X^3,0,X^1).
\]
They preserve the center $e_0$ and the geodesic ball.
Equation \eqref{eq:normal} gives
\[
\langle Z_2,\nu\rangle_L=g_0'\cos\theta,\qquad
\langle Z_3,\nu\rangle_L=g_0'\sin\theta.
\]
Indeed, $\langle Z_2,\nu\rangle_L=(-B\nu^1+(K/B)g_1)\cos\theta$ and
$\langle Z_3,\nu\rangle_L=(-B\nu^1+(K/B)g_1)\sin\theta$, while
\[
-B\nu^1+\frac K B g_1
=\frac{Bv}{A}\cosh\varphi+\frac K{AB}\sinh\varphi
=g_0'.
\]
These rotations tilt the axis of the catenoid. Rotation about
the axis itself has zero normal component. The two displayed
fields span the entire Robin kernel, as asserted.
\end{proof}

\begin{remark}
For each fixed $a>\tfrac12$, the theorem gives nondegeneracy
modulo the rotations preserving the supporting ball. All
spectral comparisons above are made for that fixed parameter;
no uniform spectral gap as the parameter approaches an
endpoint is required. In the terminology of \cite{BettiolPiccioneSantoro},
$\Sigma_a$ is equivariantly nondegenerate; see Corollary~\ref{cor:equivariant}.
\end{remark}

\section{Consequences}\label{sec:consequences}

\begin{corollary}\label{cor:radius}
The supporting radius $r(a)$ depends smoothly on $a$ and satisfies
\[
\frac{\mathrm dr}{\mathrm da}>0\qquad\text{for every }a>\tfrac12 .
\]
In particular, distinct members of the family $\{\Sigma_a\}_{a>1/2}$
are free boundary in distinct geodesic balls centered at $e_0$.
\end{corollary}

\begin{proof}
For $s>0$, the function $\tanh\varphi-K/v$ has $s$-derivative
$\varphi'\operatorname{sech}^2\varphi+Kv'/v^2>0$. By the implicit function theorem, the solution $S$ of \eqref{eq:FB} depends smoothly on $a$, and so does $r=\arcosh g_0(S)$, since $g_0(S)>1$.

All quantities of Section~\ref{sec:geometry} depend smoothly on
$(s,a)$. Let $Y=\partial_aX$, taken at fixed $(s,\theta)$, and set $\psi=\langle Y,\nu\rangle_L$. This is the parametric Jacobi field of \cite[Definition~8.1]{PigLocal}; its evenness, the value $\psi(0)=1/(2K)$ and the boundary identity \eqref{eq:radiusdefect} below are established in \cite[Lemmas~8.2 and~8.3]{PigLocal}, and we include the short
derivations for completeness. Each $X(\cdot,\cdot\,;a)$ is a minimal immersion of $\RR\times(\RR/2\pi\mathbb Z)$, so the normal component of the variation field is a Jacobi field. Since $\psi$ does not depend on $\theta$, this gives $J_0\psi=0$. The Lorentz reflection $R=\operatorname{diag}(1,-1,1,1)$ satisfies $X(-s,\theta)=RX(s,\theta)$
and $\nu(-s,\theta)=R\nu(s,\theta)$ for every $a$; hence $\psi$ is even.

At $s=0$ one has $\varphi=\partial_a\varphi=0$ and $v=0$. Since
$\partial_aA^2=\partial_aB^2=\cosh(2s)$, the vectors $Y$ and $\nu$
at $s=0$ are
\[
\Bigl(\frac1{2A},0,\frac{\cos\theta}{2B},\frac{\sin\theta}{2B}\Bigr)
\quad\text{and}\quad
\Bigl(\frac KA,0,\frac KB\cos\theta,\frac KB\sin\theta\Bigr),
\]
with $A=A(0)$ and $B=B(0)$. Using $A(0)^2-B(0)^2=1$ and $A(0)^2B(0)^2=K^2$, we obtain
\[
\psi(0)=\frac K2\Bigl(\frac1{B(0)^2}-\frac1{A(0)^2}\Bigr)
=\frac1{2K}\neq0.
\]

Now let $s=S$. The boundary relation of Section~\ref{sec:geometry}
gives $e_0=\cosh r\,X-\sinh r\,X_s$; in particular
$\langle X_s,e_0\rangle_L=-\sinh r$. Differentiating
$\langle X,X\rangle_L=-1$, $\langle\nu,X\rangle_L=0$ and
$\langle\nu,X_s\rangle_L=0$ in $a$, and using
$\nu_s=-(K/B^2)X_s$, we find
\[
\langle Y,X\rangle_L=0,\qquad
\langle\partial_a\nu,X\rangle_L=-\psi,\qquad
\langle\partial_a\nu,X_s\rangle_L
=-\psi'-\frac K{B^2}\langle Y,X_s\rangle_L .
\]
Since $\partial_ag_0=-\langle Y,e_0\rangle_L=\sinh r\,\langle Y,X_s\rangle_L$
at $s=S$ and $g_0'(S)=\sinh r$, differentiating
$g_0(S(a))=\cosh r(a)$ and dividing by $\sinh r$ gives
\[
\frac{\mathrm dS}{\mathrm da}+\langle Y,X_s\rangle_L
=\frac{\mathrm dr}{\mathrm da}.
\]
Since $\langle\nu_s,e_0\rangle_L=(K/B_0^2)\sinh r$ at $s=S$,
differentiating the free boundary condition
$\langle\nu,e_0\rangle_L|_{s=S(a)}=-\nu^0(S(a))=0$, and using the decomposition of $e_0$ together with the identities above, gives
\[
\frac K{B_0^2}\sinh r\,\frac{\mathrm dS}{\mathrm da}
-\cosh r\,\psi(S)
+\sinh r\Bigl(\psi'(S)+\frac K{B_0^2}\langle Y,X_s\rangle_L\Bigr)=0 .
\]
Dividing by $\sinh r$ and combining the last two identities, we obtain
\begin{equation}\label{eq:radiusdefect}
\psi'(S)-c\,\psi(S)=-\frac K{B_0^2}\,\frac{\mathrm dr}{\mathrm da}.
\end{equation}

Suppose that $\mathrm dr/\mathrm da=0$ for some $a$. Then
\eqref{eq:radiusdefect} and the evenness of $\psi$ show that $\psi$ satisfies both conditions in \eqref{eq:Robin}. Since $J_0\psi=0$ and $\psi(0)\neq0$, zero would be an eigenvalue of the radial Robin problem for $J_0$, contrary to Proposition~\ref{prop:modezero}. Hence
$\mathrm dr/\mathrm da$ never vanishes and has constant sign. By
Lemma~\ref{lem:boundary},
\[
\cosh r=A(S)\cosh\varphi(S)\ge\sqrt{1+B_0^2}>\sqrt{1+2K^2},
\]
and the right-hand side tends to infinity as $a\to\infty$. Thus $r$
is unbounded as $a\to\infty$, which excludes the negative sign.
\end{proof}

\begin{lemma}\label{lem:radiuslimit}
$S(a)\to0$ and $r(a)\to0$ as $a\to\tfrac12^+$.
\end{lemma}

\begin{proof}
Put $\epsilon=a-\tfrac12\in(0,\tfrac12)$. Then
$K^2=\epsilon(1+\epsilon)\in(\epsilon,2\epsilon)$ and
$\epsilon+\sinh^2 s\le B(s)^2=\epsilon+2a\sinh^2 s\le\epsilon+2\sinh^2 s$.
Fix $s_1\in(0,1]$ and suppose that $S\ge s_1$. As $v$ is increasing and
$v(s_1)\ge c_1:=\sinh(2s_1)/\bigl(2\sqrt{1/2+2\sinh^2s_1}\bigr)>0$,
\eqref{eq:FB} gives
\[
\tanh\varphi(s_1)\le\tanh\varphi(S)=\frac K{v(S)}\le\frac{\sqrt{2\epsilon}}{c_1}.
\]
On $[0,1]$ one has $A^2\le c_2:=\tfrac32+2\sinh^21$ and, since
$\sinh\sigma\le\sigma\sinh1$ there, $B(\sigma)^2\le\epsilon+c_3\sigma^2$ with
$c_3=2\sinh^21$. Hence
\[
\varphi(s_1)\ge\frac{\sqrt\epsilon}{c_2}\int_0^{s_1}
\frac{\dd\sigma}{\sqrt{\epsilon+c_3\sigma^2}}
=\frac{\sqrt\epsilon}{c_2\sqrt{c_3}}\,
\operatorname{arsinh}\Bigl(s_1\sqrt{c_3/\epsilon}\Bigr),
\]
while $\varphi(s_1)\le K\int_0^1(\epsilon+\sigma^2)^{-1/2}\dd\sigma\to0$.
Since $\tanh x\ge x/2$ on $[0,1]$, the quotient
$\tanh\varphi(s_1)/\sqrt\epsilon$ tends to $+\infty$ as $\epsilon\to0$, which
contradicts the previous display for $\epsilon$ small. Hence $S(a)<s_1$ for
$a$ close to $\tfrac12$, that is, $S(a)\to0$. Finally,
$\varphi(S)\le K\int_0^S B^{-1}\dd\sigma\le\sqrt2\,S\to0$ and
$A(S)^2=1+B(S)^2\to1$, so $\cosh r(a)=A(S)\cosh\varphi(S)\to1$.
\end{proof}

Combined with the limit $r(a)\to0$ as $a\to\frac12^+$ (Lemma~\ref{lem:radiuslimit}; see also \cite{PigRobin}), Corollary~\ref{cor:radius} shows that $a\mapsto r(a)$ is an increasing diffeomorphism of $(\frac12,\infty)$ onto $(0,\infty)$.
Thus every geodesic ball of $\HH^3$ is the supporting ball of an
isometric image of exactly one $\Sigma_a$. This was stated without proof by Li and Xiong \cite[Section~2]{LiXiong}. Corollary~\ref{cor:radius}
answers the question of the strict monotonicity of $r(a)$; see
\cite[Remark~1.3]{PigLocal}.

\begin{corollary}[Equivariant nondegeneracy and CMC deformations]
\label{cor:equivariant}
For every $a>\tfrac12$, the critical spherical catenoid
$\Sigma_a\subset\mathbb B(r(a))\subset\HH^3$ is equivariantly nondegenerate in the sense of \cite[Definition~2.4]{BettiolPiccioneSantoro}.
Consequently, there exist $\varepsilon=\varepsilon(a)>0$ and a smooth family
$\{\Sigma_{a,h}\}_{|h|<\varepsilon}$ of embedded free boundary annuli in the fixed ball $\mathbb B(r(a))$ such that $\Sigma_{a,0}=\Sigma_a$ and $\Sigma_{a,h}$ has constant mean curvature $h$; here the mean curvature is the trace of the second fundamental form with respect to the unit normal
extending $\nu$. This family is locally unique up to congruence: a free boundary surface in $\mathbb B(r(a))$ sufficiently close to $\Sigma_a$ in $C^{2,\alpha}$, with constant mean curvature $h'$ with respect to the unit
normal close to $\nu$ and with $|h'|$ sufficiently small, is congruent to $\Sigma_{a,h'}$ in the sense of
\cite[Section~2.4]{BettiolPiccioneSantoro}.
\end{corollary}

\begin{proof}
Fix $a>\tfrac12$, let $M=\mathbb B(r(a))$ carry the hyperbolic metric $g$,
and orient $\Sigma_a$ by $\nu$. We verify the hypotheses of
\cite[Theorem~5.1]{BettiolPiccioneSantoro}.

\emph{The surface.} By Section~\ref{sec:geometry},
$\Sigma_a\cap\partial M=\partial\Sigma_a$; this is condition~(a) of
\cite[Definition~2.1]{BettiolPiccioneSantoro}. Since $\nu$ is a global unit
normal, the normal bundle of $\Sigma_a$ is orientable; this is
condition~(b). Hence $\Sigma_a$ is admissible. At $s=\pm S$ the outward unit
normal of $\partial M$ is $\pm X_s$, which is tangent to $\Sigma_a$; thus
$\Sigma_a$ is normal in the sense of the same definition. Since
$H_1(M;\mathbb Z_2)=0$ and $\Sigma_a$ is a compact, connected, two-sided
surface with $\Sigma_a\cap\partial M=\partial\Sigma_a$, the set
$M\setminus\Sigma_a$ has exactly two connected components, and both have
compact closure. Each geodesic segment in $\{X^1=0\}$ from $e_0$ to
$\partial M$ meets $\Sigma_a$ exactly once, transversally, on the neck circle
$s=0$, and there $\nu$ points away from $e_0$. Hence $e_0$ and the circle
$\partial M\cap\{X^1=0\}$ lie in different components, and, $\Sigma_a$ being
connected, $\nu$ points everywhere into the component $\Omega_1$ that
contains $\partial M\cap\{X^1=0\}$. With this choice of $\Omega_1$,
condition~(c) of \cite[Definition~2.1]{BettiolPiccioneSantoro} holds, so
$\Sigma_a$ bounds a finite volume. Moreover, the unit normal used in
\cite{BettiolPiccioneSantoro}, which points into $\Omega_1$, is $\nu$; hence
the mean curvature in \cite{BettiolPiccioneSantoro}, the trace of the second
fundamental form with respect to this normal, is the one in the statement.

\emph{Killing fields.} The rotations $Z_1(X)=(0,0,-X^3,X^2)$, $Z_2$ and
$Z_3$ of Section~\ref{sec:conclusion} fix $e_0$; hence they are Killing
fields of $g$ tangent to $\partial M$, as required in
\cite[Section~2.4]{BettiolPiccioneSantoro}. Conversely, since $\HH^3$ is
simply connected and has constant curvature, a Killing field defined on a
connected open set is the restriction of a global one. If it is tangent to
$\partial M$ on a nonempty open subset of $\partial M$, then by analyticity
it is tangent to the whole sphere $\partial M$, its flow preserves
$\partial M$, and therefore it vanishes at $e_0$; that is, it lies in
$\Span\{Z_1,Z_2,Z_3\}$. As $\langle Z_1,\nu\rangle_L=0$, the space of
Killing--Jacobi fields along $\Sigma_a$ is
\[
\Span\{\langle Z_2,\nu\rangle_L,\langle Z_3,\nu\rangle_L\}
=\Span\{g_0'\cos\theta,\ g_0'\sin\theta\},
\]
which is two-dimensional because $g_0'>0$ on $(0,S]$.

\emph{Equivariant nondegeneracy.} In \cite{BettiolPiccioneSantoro} the
second fundamental form of $\partial M$ is taken with respect to its outer unit normal; for the geodesic sphere of radius $r(a)$ this gives $\II^{\partial M}(\nu,\nu)=-\coth r(a)$ along $\partial\Sigma_a$. Hence the linearized free boundary condition \cite[(2.3)]{BettiolPiccioneSantoro} is the Robin condition $\partial_\eta u=cu$ of Section~\ref{sec:spectral}, and the space of Jacobi fields in $C^{2,\alpha}_\partial(\Sigma_a)$ is the Robin kernel of Theorem~\ref{thm:main}. By that theorem, it coincides with the
space of Killing--Jacobi fields. Thus $\Sigma_a$ is equivariantly nondegenerate.

\emph{Deformation.} Apply \cite[Theorem~5.1]{BettiolPiccioneSantoro} with $\Lambda=\RR$, the constant family $g_\lambda\equiv g$, so that all $g_\lambda$ define the same normal bundle of $\partial M$, an arbitrary
$\lambda_0\in\Lambda$, $h_0=0$, and the constant frame of Killing--Jacobi field generators $K_1^\lambda=Z_2$, $K_2^\lambda=Z_3$. The theorem gives a smooth family $\Sigma_{(h,\lambda)}$, for $(h,\lambda)$ near $(0,\lambda_0)$,
of free boundary hypersurfaces of $M$ diffeomorphic to $\Sigma_a$, with constant mean curvature $h$ and $\Sigma_{(0,\lambda_0)}=\Sigma_a$, which is unique modulo congruence. Setting $\Sigma_{a,h}=\Sigma_{(h,\lambda_0)}$ for
$|h|<\varepsilon$ proves the existence statement. Since $g_\lambda$ does not depend on $\lambda$, the uniqueness clause of \cite[Theorem~5.1]{BettiolPiccioneSantoro} with $\lambda'=\lambda_0$ is the asserted local uniqueness.
\end{proof}

\begin{remark}\label{rem:BPS}
For the critical catenoid in the Euclidean unit ball, equivariant nondegeneracy and a two-parameter family of free boundary CMC annuli in the unit balls of the space forms of curvature $\lambda$ close to $0$ were obtained by Bettiol, Piccione and Santoro \cite[Lemma~8.1 and Proposition~8.3]{BettiolPiccioneSantoro}; see also \cite[Remark~8.2]{BettiolPiccioneSantoro}. After rescaling, the members with
$\lambda<0$ of that family lie in geodesic balls of $\HH^3$ of small radius. Corollary~\ref{cor:equivariant} applies to every $a>\tfrac12$ and hence, by Corollary~\ref{cor:radius} and the discussion following it, to the critical spherical catenoid in every geodesic ball of $\HH^3$. Embedded, non-rotational free boundary CMC annuli in geodesic balls of $\HH^3$, with
mean curvature $H>1$ in their normalization, were constructed by Cerezo, Fern\'andez and Mira \cite{CerezoFernandezMira}.
\end{remark}

\begin{remark}\label{rem:covers}
The restriction to angular period $2\pi$ in Theorem~\ref{thm:main} cannot be dropped. The radial form \eqref{eq:qrad} is defined for every real $k\ge0$, and
\[
\qk_k[u]-\qk_\ell[u]=(k^2-\ell^2)\int_{-S}^{S}\frac{u^2}{B}\dd s .
\]
Since $B(0)\le B\le B_0$ on $[-S,S]$, the min--max principle gives, for the radial Robin eigenvalues $\lambda_j(k)$,
\[
\frac{k^2-\ell^2}{B_0^2}\le\lambda_j(k)-\lambda_j(\ell)\le\frac{k^2-\ell^2}{B(0)^2}
\qquad(0\le\ell<k).
\]
Hence each $\lambda_j(k)$ is a continuous, strictly increasing function of $k\ge0$.
Propositions~\ref{prop:modeone} and~\ref{prop:higher} give
$\lambda_1(1)<0<\lambda_1(2)$, so there is a unique $k_*(a)\in(1,2)$
with $\lambda_1(k_*(a))=0$. For $0<k<1$, Propositions~\ref{prop:modezero} and~\ref{prop:modeone} give
$\lambda_1(k)<\lambda_2(k)<\lambda_2(1)=0<\lambda_3(0)<\lambda_3(k)$;
for $1<k<2$, one has $\lambda_2(k)>\lambda_2(1)=0$. On the $m$-fold cover of $\Sigma_a$, of angular period $2\pi m$, the angular frequencies are $j/m$ with integer $j\ge0$. Each frequency $j/m\in(0,1)$ therefore contributes four negative directions, so for $m\ge2$ the index of the cover is at least $4m$. Moreover, the nullity of the cover exceeds two exactly when $mk_*(a)$ is an integer, which again requires $m\ge2$. This agrees with the construction of Cerezo~\cite{Cerezo},
in which non-rotational free boundary minimal annuli in geodesic balls of $\HH^3$ bifurcate from $m$-fold covers of critical spherical catenoids (called free boundary hyperbolic catenoids in \cite{Cerezo}). The bifurcating annuli are invariant under $n$ reflections in planes through the axis, where $m$ and $n$ are coprime and $n/m$ lies in an interval contained in $(\sqrt2,\sqrt3)\subset(1,2)$.
\end{remark}

\begin{remark}[Scope of the argument]\label{rem:scope}
Lemmas~\ref{lem:boundary} and~\ref{lem:logderivative} rest on identities specific to the rotational family of Section~\ref{sec:geometry}, and the separation of variables uses its rotational symmetry. Three ingredients nevertheless have a wider scope.

(i) The identity \eqref{eq:normalLap} uses neither the explicit parametrization nor the rotational symmetry: for every two-sided minimal hypersurface $\Sigma^{n}\subset\HH^{n+1}\subset\RR^{n+1,1}$, the Gauss--Weingarten and Codazzi equations give $\Delta_\Sigma\nu=-|\II|^{2}\nu$, hence $J\nu^{0}=n\,\nu^{0}$ for
$J=-\Delta_\Sigma-|\II|^{2}+n$. If $\Sigma$ is free boundary in a geodesic ball centered at $e_{0}$, then $\nu^{0}=0$ on $\partial\Sigma$. What is specific to the present family is
the sign property of Lemma~\ref{lem:normal}: $\nu^{0}$ is positive in the interior, hence it is the first Dirichlet eigenfunction, with eigenvalue $2$; this is the input of
Proposition~\ref{prop:modezero}.

(ii) The last step of Proposition~\ref{prop:higher} is an instance of the ground-state representation for regular Sturm--Liouville problems. Write the Robin conditions as $\partial_\eta z=\sigma z$ at the endpoints. If a positive solution $u$ of the zero-energy equation satisfies $\partial_\eta u/u>\sigma$ at both endpoints, the substitution $z=uh$ writes the quadratic form as a nonnegative interior square plus boundary terms whose coefficients are positive multiples of $\partial_\eta u/u-\sigma$; the form is therefore positive definite. By symmetry the two conditions reduce here to
$p_{2}(T)>v_{0}$, which is Lemma~\ref{lem:logderivative}.

(iii) The proof of Corollary~\ref{cor:radius} reflects a general transversality principle. Let $X_a$ be a smooth family of free boundary minimal hypersurfaces in the geodesic balls of radius $r(a)$ centered at $e_{0}$, parametrized over a fixed manifold with boundary, and let $\psi=\langle\partial_aX_a,\nu\rangle_L$. Then $\psi$ is a Jacobi field, and differentiating the two free boundary conditions gives
\[
\partial_\eta\psi-\coth r\,\psi=-\II(\eta,\eta)\,\frac{\mathrm dr}{\mathrm da}
\qquad\text{on }\partial\Sigma,
\]
where $\eta$ is the outward conormal; here one uses that $\eta$ is a principal direction along $\partial\Sigma$, the supporting sphere being umbilic. This is the Robin defect identity of \cite[Theorem~3.1]{PigCaps}, which holds in all space forms and dimensions, with variants for warped products and capillary hypersurfaces \cite[Remarks~3.2 and~3.3]{PigCaps}; there the second fundamental form is $-\II$. For $\Sigma_a$ one has $\II(\eta,\eta)=K/B_0^{2}$, and this is \eqref{eq:radiusdefect}. Hence, if $\mathrm dr/\mathrm da$ vanished, $\psi$ would lie in the Robin kernel. For $\Sigma_a$,
$\psi$ is nonzero since $\psi(0)=1/(2K)$, and Proposition~\ref{prop:modezero} excludes a radial Robin kernel, so $\mathrm dr/\mathrm da\neq0$; the sign is then fixed by the
unboundedness of $r$.
\end{remark}

\section{The energy index}\label{sec:energy}

Fix $a>\tfrac12$ and write $\Sigma=\Sigma_a$, $r=r(a)$ and $c=\coth r$, as in \eqref{eq:boundarynotation}. The flat connection $D$ of $\RR^{3,1}$ acts componentwise on $\RR^{3,1}$-valued functions along $\Sigma$; for such functions $U,V$ set
\[
\langle\nabla U,\nabla V\rangle_L=\sum_{i=1}^{2}\langle D_{e_i}U,D_{e_i}V\rangle_L
\]
for a local orthonormal frame $\{e_1,e_2\}$ of $T\Sigma$. We write $\hat X=(X^1,X^2,X^3)$, so that $X=(X^0,\hat X)$, and $dX(W)$ for the image of a tangent vector field $W$ of $\Sigma$, as in the Gauss formula of Section~\ref{sec:boundaryinequality}. The shape operator $\mathcal A$ satisfies $D\nu=-dX\circ\mathcal A$ and $\II(U,V)=g(\mathcal AU,V)$. As in Section~\ref{sec:spectral}, $\eta$ is the outward unit conormal along $\partial\Sigma$. By Section~\ref{sec:geometry} and the reflection symmetry,
\begin{equation}\label{eq:etaboundary}
\eta=\frac{\cosh r\,X-e_0}{\sinh r}=cX-\frac{e_0}{\sinh r}
\qquad\text{along }\partial\Sigma .
\end{equation}

\subsection{Energy index and spectral index}

Let
\[
\mathcal V=\bigl\{V\in C^\infty(\Sigma;\RR^{3,1}) :
\langle V,X\rangle_L=0 \text{ on } \Sigma,\ \
\langle V,e_0\rangle_L=0 \text{ on } \partial\Sigma\bigr\},
\]
and
\begin{equation}\label{eq:SE}
S_E(U,V)=\int_\Sigma\bigl(\langle\nabla U,\nabla V\rangle_L
+2\langle U,V\rangle_L\bigr)\dd A
-c\int_{\partial\Sigma}\langle U,V\rangle_L\dd\ell .
\end{equation}
The energy index $\ind_E(\Sigma)$ is the maximal dimension of a subspace of $\mathcal V$ on which $S_E$ is negative definite. Since $S_E$ depends only on the restriction to $\Sigma$, and every $V\in\mathcal V$ extends to a smooth field on $\mathbb B(r)$ tangent to $\HH^3$ whose $e_0$-component vanishes on $\partial\mathbb B(r)$, this is the energy index of \cite[Definition~2]{Medvedev}. To construct the extension, let $V_1$ be the orthogonal projection onto $T\HH^3$ of a smooth extension of $V$ to $\mathbb B(r)$; then $V_1=V$ on $\Sigma$. Let $\rho$ be the distance from $e_0$, and let $\chi$ be a cutoff function that vanishes near $e_0$ and equals one near $\partial\mathbb B(r)$. Since $\Sigma$ meets $\partial\mathbb B(r)$ transversally along $\partial\Sigma$, where $\langle V_1,e_0\rangle_L=0$, there is a smooth function $h$ on $\mathbb B(r)$ that vanishes on $\Sigma$ and equals $\langle V_1,e_0\rangle_L$ on $\partial\mathbb B(r)$. As $\langle\partial_\rho,e_0\rangle_L=-\sinh\rho$, the field $V_1+\chi h\,(\sinh\rho)^{-1}\partial_\rho$ is the required extension.

The spectral index $\ind_S(\Sigma)$ is the number, counted with multiplicity, of eigenvalues $\sigma<c$ of the Steklov-type problem $\Delta_\Sigma u=2u$ in $\Sigma$, $\partial_\eta u=\sigma u$ on $\partial\Sigma$ \cite[Definition~3]{Medvedev}. By the Dirichlet principle and the positivity of $\mathcal R$ on $H^1_0(\Sigma)$, it is also the index on $H^1(\Sigma)$ of the form
\begin{equation}\label{eq:R}
\mathcal R(u,w)=\int_\Sigma\bigl(\langle\nabla u,\nabla w\rangle+2uw\bigr)\dd A-c\int_{\partial\Sigma}uw\dd\ell .
\end{equation}
For $\Sigma_a$ it equals one \cite[Theorem~18]{Medvedev}. This also follows from Section~\ref{sec:spectral}. After separation of variables, the radial forms of $\mathcal R$ are those in \eqref{eq:qrad} without the term $-2K^2u^2/B^3$, and the Sturm--Liouville facts recalled there apply. For $k=0$, the function $g_1$ solves the radial equation by \eqref{eq:gode} and satisfies \eqref{eq:Robin} by \eqref{eq:boundarycoordinates}, since it is odd. It vanishes only at $s=0$, because $g_1=A\sinh\varphi$ and $\varphi$ is odd and strictly increasing. Hence $g_1$ is the second eigenfunction, with eigenvalue zero, and the mode $k=0$ contributes exactly one negative eigenvalue. For $k=1$, the function $B$ solves the radial equation, since $X^2=B\cos\theta$ and $\Delta_\Sigma X^2=2X^2$, and it satisfies \eqref{eq:Robin}, since $v_0=cB_0$ by \eqref{eq:boundarycoefficient} and $v$ is odd. As $B>0$, the lowest eigenvalue for $k=1$ is zero. For $k\ge2$ the radial forms exceed the one for $k=1$ and are therefore positive. Hence $\ind_S(\Sigma_a)=1$.

\begin{remark}\label{rem:SE}
The form \eqref{eq:SE} is formula~(9) in \cite{Medvedev} with $n=3$ and $k=2$. There $|\nabla V|^2=-|\nabla V^0|^2+\sum_{j\ge1}|\nabla V^j|^2$, and the conditions defining $\mathcal V$ are those of \cite[Remark~5]{Medvedev}. The norm $|V|^2$ is not written out in \cite{Medvedev}; we take $|V|^2=\langle V,V\rangle_L$, the squared hyperbolic length of $V\in T_X\HH^3$. With this choice the formula is the second variation of the energy, as stated in \cite{Medvedev}; with the Euclidean norm of $\RR^4$ the resulting form would exceed $S_E$ by $4\int_\Sigma(V^0)^2\dd A$ on $\mathcal V$. For completeness we check that $S_E$ is the second variation of the energy $E=\frac12\int_\Sigma\langle\nabla X,\nabla X\rangle_L\dd A$, the conformal structure of $\Sigma$ being fixed. Let $X_t$ be a smooth variation with $X_t(\Sigma)\subset\HH^3$, $X_t(\partial\Sigma)\subset\partial\mathbb B(r)$, $X_0=X$ and $\partial_tX_t|_{t=0}=V$, and put $Z=\partial_t^2X_t|_{t=0}$. Differentiating $\langle X_t,X_t\rangle_L=-1$, and $\langle X_t,e_0\rangle_L=-\cosh r$ on $\partial\Sigma$, twice gives $\langle Z,X\rangle_L=-\langle V,V\rangle_L$ and $\langle Z,e_0\rangle_L=0$ on $\partial\Sigma$. Since $\Delta_\Sigma X=2X$ (Section~\ref{sec:boundaryinequality}) and $\partial_\eta X=\eta$ is given by \eqref{eq:etaboundary}, Green's formula gives
\[
\frac{\mathrm d^2E}{\mathrm dt^2}\Big|_{t=0}
=\int_\Sigma\langle\nabla V,\nabla V\rangle_L\dd A
-\int_\Sigma\langle\Delta_\Sigma X,Z\rangle_L\dd A
+\int_{\partial\Sigma}\langle\partial_\eta X,Z\rangle_L\dd\ell
=S_E(V,V).
\]
\end{remark}

\subsection{Energy and area}

\begin{lemma}\label{lem:energyarea}
Every $V\in\mathcal V$ can be written uniquely as $V=u\nu+dX(W)$, where $u\in C^\infty(\Sigma)$ and $W$ is a smooth vector field on $\Sigma$ tangent to $\partial\Sigma$ along $\partial\Sigma$; conversely, every such field lies in $\mathcal V$. For $V=u\nu+dX(W)$ and $V'=u'\nu+dX(W')$ in $\mathcal V$,
\begin{equation}\label{eq:energyarea}
S_E(V,V')=Q[u,u']
+\frac12\int_\Sigma\bigl\langle 2u\,\II-(\mathcal L_Wg)^\circ,\,
2u'\,\II-(\mathcal L_{W'}g)^\circ\bigr\rangle\dd A,
\end{equation}
where $Q[\cdot,\cdot]$ is the symmetric bilinear form associated with \eqref{eq:Q} and $(\cdot)^\circ$ denotes the trace-free part of a symmetric $2$-tensor.
\end{lemma}

\begin{proof}
Since $\langle V,X\rangle_L=0$, the field $V$ is tangent to $\HH^3$; put $u=\langle V,\nu\rangle_L$ and let $dX(W)$ be its tangential part. By \eqref{eq:etaboundary}, $e_0=\cosh r\,X-\sinh r\,\eta$ along $\partial\Sigma$, so $\langle V,e_0\rangle_L=-\sinh r\,g(W,\eta)$ there. Thus the boundary condition in $\mathcal V$ is equivalent to $W\perp\eta$.

By polarization it suffices to take $V=V'$. By the Gauss formula,
\[
D_{e_i}\nu=-dX(\mathcal Ae_i),\qquad
D_{e_i}(dX(W))=dX(\nabla_{e_i}W)+\II(e_i,W)\,\nu+g(e_i,W)\,X .
\]
Since $X$, $\nu$ and $dX(T\Sigma)$ are mutually orthogonal, with $\langle\nu,\nu\rangle_L=1$ and $\langle X,X\rangle_L=-1$,
\begin{align*}
\langle\nabla(u\nu),\nabla(u\nu)\rangle_L&=|\nabla u|^2+|\II|^2u^2,\\
\langle\nabla(dX(W)),\nabla(dX(W))\rangle_L&=|\nabla W|^2+|\mathcal AW|^2-|W|^2,\\
\langle\nabla(u\nu),\nabla(dX(W))\rangle_L&=\II(\nabla u,W)-u\langle\II,\nabla W\rangle .
\end{align*}

(a) By \eqref{eq:Q}, $S_E(u\nu,u\nu)=Q[u]+2\int_\Sigma|\II|^2u^2\dd A$.

(b) The shape operator has eigenvalues $\pm K/B^2$ (Section~\ref{sec:geometry}), so $|\mathcal AW|^2=\frac12|\II|^2|W|^2$, and the Gauss equation gives the Gauss curvature $\mathcal K=-1-\frac12|\II|^2$ of $\Sigma$. Hence
\[
S_E(dX(W),dX(W))=\int_\Sigma\bigl(|\nabla W|^2-\mathcal K|W|^2\bigr)\dd A
-c\int_{\partial\Sigma}|W|^2\dd\ell .
\]
On the other hand, for every vector field $W$ on a surface,
\[
\tfrac12|(\mathcal L_Wg)^\circ|^2=|\nabla W|^2-\mathcal K|W|^2
+\operatorname{div}\bigl(\nabla_WW-(\operatorname{div}W)\,W\bigr),
\]
as follows from $\frac12|(\mathcal L_Wg)^\circ|^2=\frac12|\mathcal L_Wg|^2-(\operatorname{div}W)^2$ and the Ricci identity $\nabla_i\nabla_jW^i=\nabla_j\operatorname{div}W+\mathcal K W_j$. If $W=w\tau$ on $\partial\Sigma$, with $\tau$ a unit tangent of $\partial\Sigma$, then
\[
g\bigl(\nabla_WW-(\operatorname{div}W)W,\eta\bigr)
=w^2g(\nabla_\tau\tau,\eta)
=-w^2\bigl\langle\tau,\nabla^{\HH^3}_\tau\eta\bigr\rangle_L=-c\,w^2,
\]
because $\eta$ is the outward unit normal of the supporting sphere (Section~\ref{sec:spectral}). Integrating, $S_E(dX(W),dX(W))=\frac12\int_\Sigma|(\mathcal L_Wg)^\circ|^2\dd A$.

(c) As in the proof of Lemma~\ref{lem:normal}, the Codazzi equation and minimality give $\operatorname{div}\mathcal A=0$, hence $\operatorname{div}\bigl(\mathcal AW\bigr)=\langle\II,\nabla W\rangle$. On $\partial\Sigma$ one has $\II(\tau,\eta)=\langle\nu,D_\tau\eta\rangle_L=c\,\langle\nu,\tau\rangle_L=0$, since the supporting sphere is umbilic. Integration by parts therefore gives $\int_\Sigma\II(\nabla u,W)\dd A=-\int_\Sigma u\langle\II,\nabla W\rangle\dd A$. The terms of order zero vanish because $\langle\nu,dX(W)\rangle_L=0$, and so
\[
S_E(u\nu,dX(W))=-2\int_\Sigma u\langle\II,\nabla W\rangle\dd A
=-\int_\Sigma u\,\bigl\langle\II,(\mathcal L_Wg)^\circ\bigr\rangle\dd A,
\]
where in the last step we used that $\II$ is symmetric and trace-free.

Adding (a), twice (c), and (b) gives \eqref{eq:energyarea} for $V=V'$.
\end{proof}

\subsection{Conical fields}

\begin{lemma}\label{lem:conical}
Let $C$ be a real $3\times3$ matrix, $P=(0,C\hat X)$, $f_C=\langle P,X\rangle_L=\hat X\cdot C\hat X$ and $V_C=P+f_CX$. Then $\langle V_C,X\rangle_L=0$. If $f_C=0$ on $\partial\Sigma$, then $V_C\in\mathcal V$ and
\[
S_E(V_C,V_C)=-\int_\Sigma|\nabla f_C|^2\dd A .
\]
\end{lemma}

\begin{proof}
We have $\langle V_C,X\rangle_L=f_C+f_C\langle X,X\rangle_L=0$, and $\langle V_C,e_0\rangle_L=-f_CX^0$ vanishes on $\partial\Sigma$ when $f_C$ does. By \eqref{eq:SE} and \eqref{eq:R}, $S_E(U,V)=-\mathcal R(U^0,V^0)+\sum_{j=1}^3\mathcal R(U^j,V^j)$. For $j=1,2,3$ one has $\Delta_\Sigma X^j=2X^j$ in $\Sigma$, and $\partial_\eta X^j=cX^j$ on $\partial\Sigma$ by the spatial components of \eqref{eq:etaboundary}. By Green's formula, $\mathcal R(X^j,w)=0$ for every $w\in H^1(\Sigma)$. Since $P^0=0$ and each $P^j$ is a linear combination of $X^1,X^2,X^3$, it follows that $S_E(P,V)=0$ for every $V$. Therefore $S_E(V_C,V_C)=S_E(f_CX,f_CX)$. Finally $D_{e_i}(f_CX)=(e_if_C)X+f_C\,dX(e_i)$, so $\langle\nabla(f_CX),\nabla(f_CX)\rangle_L=-|\nabla f_C|^2+2f_C^2$, while $\langle f_CX,f_CX\rangle_L=-f_C^2$. Hence
\[
S_E(f_CX,f_CX)=\int_\Sigma\bigl(-|\nabla f_C|^2+2f_C^2-2f_C^2\bigr)\dd A
+c\int_{\partial\Sigma}f_C^2\dd\ell=-\int_\Sigma|\nabla f_C|^2\dd A. \qedhere
\]
\end{proof}

\subsection{Proof of Theorem~\ref{thm:energy}}

\begin{proof}[Proof of Theorem~\ref{thm:energy}]
\emph{Upper bound.} Let $U\subset\mathcal V$ be a subspace on which $S_E$ is negative definite. By Lemma~\ref{lem:energyarea}, $S_E(V,V)\ge Q[\langle V,\nu\rangle_L]$ for $V\in\mathcal V$; moreover, if $\langle V,\nu\rangle_L=0$, then $S_E(V,V)=\frac12\int_\Sigma|(\mathcal L_Wg)^\circ|^2\dd A\ge0$. Hence the map $V\mapsto\langle V,\nu\rangle_L$ is injective on $U$, and $Q$ is negative definite on its image. By Theorem~\ref{thm:main}, $\dim U\le4$.

\emph{Lower bound: the conical field.} Let $\mu=B_0^2/g_1(S)^2$, which is well defined because $g_1(S)=A(S)\sinh\varphi(S)>0$, and let $C=\operatorname{diag}(-\mu,1,1)$. By \eqref{eq:X}, $f_C=B^2-\mu\,g_1^2$, a function of $s$ only. As $B$ is even and $g_1$ is odd, $f_C$ is even, and $f_C(\pm S)=0$. Since $f_C(0)=B(0)^2>0$, the function $f_C$ is not constant. Lemma~\ref{lem:conical} gives $V_C\in\mathcal V$ and
\[
S_E(V_C,V_C)=-2\pi\int_{-S}^{S}f_C'(s)^2\,B(s)\dd s<0 .
\]
Moreover, by \eqref{eq:X} and \eqref{eq:normal}, $\langle V_C,\nu\rangle_L=\langle P,\nu\rangle_L=K-\mu\,g_1\nu^1$, which depends only on $s$ and is even.

\emph{Lower bound: three fields from Theorem~\ref{thm:main}.} By Proposition~\ref{prop:modezero}, there is an odd function $y_0$ satisfying \eqref{eq:Robin} with $J_0y_0=\lambda y_0$ and $\lambda<0$. By Proposition~\ref{prop:modeone}, there is an even function $y_1$ satisfying \eqref{eq:Robin} with $J_1y_1=\lambda'y_1$ and $\lambda'<0$. Put $\zeta_1=y_0(s)$, $\zeta_2=y_1(s)\cos\theta$ and $\zeta_3=y_1(s)\sin\theta$; then $Q[\zeta_i]<0$. By \eqref{eq:metric} and the formulas for $\nu_s$ and $\nu_\theta$ in Section~\ref{sec:geometry}, $\II=(K/B^2)\,ds^2-K\,d\theta^2$. For $W=W^s\partial_s+W^\theta\partial_\theta$, the identity $(\mathcal L_Wg)^\circ=2u\,\II$ is equivalent to
\begin{equation}\label{eq:CR}
B\,\partial_s\Bigl(\frac{W^s}{B}\Bigr)-\partial_\theta W^\theta=\frac{2Ku}{B^2},
\qquad
\partial_\theta W^s+B^2\,\partial_sW^\theta=0 .
\end{equation}
For $\zeta_1$ take $W^\theta=0$ and $W^s=B\gamma$, with
\[
\gamma(s)=\int_{-S}^{s}\frac{2K\,y_0(\sigma)}{B(\sigma)^3}\dd\sigma ;
\]
then $\gamma(\pm S)=0$, because the integrand is odd. For $\zeta_2$ take $W^s=B\alpha(s)\cos\theta$ and $W^\theta=\beta(s)\sin\theta$, where $\alpha$ solves
\[
B\,(B\alpha')'-\alpha=B\Bigl(\frac{2Ky_1}{B^2}\Bigr)'\quad\text{on }(-S,S),\qquad \alpha(\pm S)=0,
\]
and $\beta=B\alpha'-2Ky_1/B^2$. This boundary value problem is uniquely solvable: in the conformal coordinate $t$ of Section~\ref{sec:highermodes} it reads $\ddot\alpha-\alpha=\frac{\mathrm d}{\mathrm dt}\bigl(2Ky_1/B^2\bigr)$, with Dirichlet conditions at $t=\pm T$. For $\zeta_3$ take $W^s=B\alpha(s)\sin\theta$ and $W^\theta=-\beta(s)\cos\theta$. In each case \eqref{eq:CR} holds with $u=\zeta_i$, and $W^s(\pm S)=0$, so the field $W_i$ is tangent to $\partial\Sigma$. Put $V_i=\zeta_i\nu+dX(W_i)\in\mathcal V$. Since $2\zeta_i\II-(\mathcal L_{W_i}g)^\circ=0$, Lemma~\ref{lem:energyarea} gives
\[
S_E(V_i,V)=Q\bigl[\zeta_i,\langle V,\nu\rangle_L\bigr]\qquad\text{for every }V\in\mathcal V .
\]

\emph{Orthogonality.} The form $Q$ is invariant under the rotations about the axis and under the reflection $s\mapsto-s$. Hence $Q[\zeta_i,\zeta_j]=0$ for $i\ne j$, and $Q[\zeta_i,\langle V_C,\nu\rangle_L]=0$ for $i=1,2,3$, because $\langle V_C,\nu\rangle_L$ is radial and even, whereas $\zeta_1$ is radial and odd and $\zeta_2,\zeta_3$ have angular frequency one. Thus the matrix of $S_E$ on $V_C,V_1,V_2,V_3$ is diagonal with negative entries, and $\ind_E(\Sigma_a)\ge4$. With the upper bound, $\ind_E(\Sigma_a)=4$. Since $\ind_S(\Sigma_a)=1$ \cite[Theorem~18]{Medvedev}, we obtain $\ind_E(\Sigma_a)=4>3=3\ind_S(\Sigma_a)$.
\end{proof}

\begin{remark}\label{rem:curvature}
The proof of Lemma~\ref{lem:conical} uses only $\Delta_\Sigma X=2X$ and the Robin condition for the spatial coordinates. It therefore applies to every free boundary minimal surface in a geodesic ball of $\HH^3$ whose boundary lies on a cone $\{\hat X\cdot C\hat X=0\}$ and which is not contained in it.

The same computation in a geodesic ball of $\mathbb S^3_+$, with the Euclidean product and formula~(8) in \cite{Medvedev}, gives $V_C=P-f_CX$ and $S_E(V_C,V_C)=+\int_\Sigma|\nabla f_C|^2\dd A$. In the Euclidean unit ball one has $V_C=Cx$ and $S_E(V_C,\cdot)=0$ on admissible fields. For the critical catenoid in the unit ball of $\RR^3$, with the cone through its boundary circles, $Cx$ is thus a null direction of the energy that is not induced by a rotation. Positive curvature moves this direction to the positive side, consistently with \cite[Proposition~11]{Medvedev}; negative curvature moves it to the negative side. In particular, the route of \cite[Theorem~12]{Medvedev}, namely $\ind\le n\ind_S+\dim\mathcal M$, where $\mathcal M$ is the moduli space of conformal structures on $\Sigma$ (of dimension one for an annulus) and the inequality $\ind\le\ind_E+\dim\mathcal M$ is due to V.~Lima \cite{LimaBounds}, cannot give $\ind(\Sigma_a)\le4$ in $\HH^3$: by Theorem~\ref{thm:energy}, already $\ind_E(\Sigma_a)=\ind(\Sigma_a)=4$.

As in the proof of Lemma~\ref{lem:conical}, $S_E(V,V)=-\mathcal R(V^0,V^0)+\sum_{j=1}^3\mathcal R(V^j,V^j)$, and for $V\in\mathcal V$ the first term equals $-\int_\Sigma\bigl(|\nabla V^0|^2+2(V^0)^2\bigr)\dd A\le0$, because $V^0=0$ on $\partial\Sigma$. In formula~(8) of \cite{Medvedev} all components enter with the same sign, and the proof of \cite[Proposition~11]{Medvedev} bounds $S_E$ from below by the spectral forms of the components. The opposite sign of the time component in $\HH^3$ is the difficulty noted in \cite[Remark~7]{Medvedev}.
\end{remark}

\end{document}